\documentclass[a4paper]{scrartcl}

\usepackage{amsthm}
\usepackage{amsmath}
\usepackage{amssymb}
\usepackage[all,cmtip]{xy}
\usepackage{color}
\usepackage[T1]{fontenc}
\usepackage{mathrsfs}
\usepackage{enumerate}

\theoremstyle{plain}
\newtheorem{thm}{Theorem}

\newtheorem{prop}[thm]{Proposition}

\theoremstyle{definition}
\newtheorem{defn}[thm]{Definition}
\newtheorem{exmp}[thm]{Example}

\newtheorem{rem}[thm]{Remark}

\newtheorem*{conventions}{Conventions}

\DeclareMathOperator{\ob}{ob}
\DeclareMathOperator{\mor}{mor}

\DeclareMathOperator{\Aut}{Aut}

\DeclareMathOperator{\id}{id}
\DeclareMathOperator{\Sep}{Sep}
\DeclareMathOperator{\Per}{Per}
\DeclareMathOperator{\Supp}{Supp}
\DeclareMathOperator{\Ann}{Ann}

\DeclareMathOperator{\Top}{Top}
\DeclareMathOperator{\Ring}{Ring}
\DeclareMathOperator{\op}{op}

\DeclareMathOperator{\opsqr}{sqr}
\DeclareMathOperator{\opsqrt}{sqrt}
\DeclareMathOperator{\opabs}{abs}

\DeclareMathOperator{\dom}{Dom}
\DeclareMathOperator{\cod}{Cod}
\DeclareMathOperator{\orb}{Orb}

\begin{document}

\title{Skew Category Algebras Associated with Partially Defined Dynamical Systems}
\date{}

\author{Patrik Lundstr\"{o}m\footnote{Address: University West, Department of Engineering Science, SE-46186 Trollh\"{a}ttan, Sweden, E-mail: Patrik.Lundstrom@hv.se, Fax: +46-520-223099} \and Johan \"{O}inert\footnote{Address: Department of Mathematical Sciences, University of Copenhagen, Universitetsparken 5, DK-2100 Copenhagen \O, Denmark, E-mail: oinert@math.ku.dk, Fax: +45-35320704}}

%
%


\maketitle

\begin{abstract}

We introduce partially defined dynamical systems defined
on a topological space. To each such system we associate
a functor $s$ from a category $G$ to $\Top^{\op}$ and show
that it defines what we call a skew category algebra
$A \rtimes^{\sigma} G$. We study the connection between
topological freeness of $s$ and, on the one hand,
ideal properties of $A \rtimes^{\sigma} G$
and, on the other hand, maximal commutativity
of $A$ in $A \rtimes^{\sigma} G$.
In particular, we show that
if $G$ is a groupoid and
for each $e \in \ob(G)$
the group of all morphisms $e \rightarrow e$
is countable and the topological space $s(e)$
is Tychonoff and Baire, then the following
assertions are equivalent: (i) $s$ is topologically free;
(ii) $A$ has the ideal intersection property, that is
if $I$ is a nonzero ideal of $A \rtimes^{\sigma} G$,
then $I \cap A \neq \{ 0 \}$;
(iii) the ring $A$ is a maximal abelian
complex subalgebra of $A \rtimes^{\sigma} G$.
Thereby, we generalize a result
by Svensson, Silvestrov and de Jeu from
the addi\-tive group of integers to a large class of groupoids.\\[4pt]
\end{abstract}


\pagestyle{headings}

\section{Introduction}

We have known for a long time that there is a connection between
topological properties of spaces and algebraical properties of
rings. In the classical papers \cite{mur36}, \cite{mur43} and \cite{von61}
by von Neumann and Murray a link between ergodic theory and the theory of von
Neumann algebras was established. Namely, they show that if the action is free, then the associated crossed
product is a factor if and only if the action is ergodic. They also
give precise conditions on the measure-theoretic side under which it
is a factor of a certain type. A similar well-known result is the
theorem of Krieger (see \cite{krieger69} and \cite{krieger76}) saying that two such ergodic group
actions are orbit equivalent if and only if their associated crossed
product von Neumann algebras are isomorphic. This strong ergodic
interplay has stimulated the study of a more general topological
situation in the following sense.
Suppose that $X$ is a topological space and
$s : X \rightarrow X$ is a continuous function;
in that case the pair $(X,s)$ is called a
\emph{topological dynamical system}.
An element $x \in X$ is called
\emph{periodic} if there is a positive integer $n$
such that $s^n(x)=x$; an element of $X$ which is not periodic
is called \emph{aperiodic}.
Recall that the topological dynamical system
$(X,s)$ is called {\it topologically free}
if the set of aperiodic elements of $X$ is dense in $X$.
Note that if $X$ carries the discrete topology,
then $(X,s)$ is topologically free if and only
if $(X,s)$ is {\it algebraically free}, that is if for every $x \in X$
and any positive integer $n$ we have that $s^n(x) \neq x$.
Now suppose that $s$ is a homeomorphism
of a compact and Hausdorff topological space $X$.
Denote by $C(X)$ the unital $C^*$-algebra of continuous complex-valued functions
on $X$ endowed with the supremum norm, conjugation as involution
and pointwise addition and multiplication.
The map $\sigma_s : C(X) \rightarrow C(X)$
which to a function $f \in C(X)$ associates $f \circ s \in C(X)$
is then an automorphism of $C(X)$.
The action of $\sigma_s$ on $C(X)$ extends in a unique
way to a strongly continuous representation $\sigma : {\Bbb Z} \rightarrow \Aut(C(X))$
subject to the condition that $\sigma(1) = \sigma_s$,
namely $\sigma(k) = \sigma_s^k$, for $k \in {\Bbb Z}$.
In that case, the associated transformation group
$C^*$-algebra $C^*(X,s)$ can be constructed
(see e.g. \cite{bla06} or \cite{pedersen79} for the details).
It has been observed,
in the works of Zeller-Meier \cite{zel68}, Effros, Hahn \cite{effros67},
Elliott \cite{elliot80}, Archbold,
Quigg, Spielberg \cite{archbold93, quigg92, spielberg91},
Kishimoto, Kawamura, Tomiyama \cite{kishimoto81, kawamura90, tomiyama92},
that topological freness of $(X,s)$ is closely linked with,
on the one hand, ideal properties of $C^*(X,s)$ and,
on the other hand, commutativity properties of $C(X)$ in $C^*(X,s)$.
The main impetus for this article is the following elegant theorem
(formulated as Theorem 4.5 in the book \cite{tomiyama92} by Tomiyama)
summarizing such results from \cite{archbold93, elliot80, kawamura90,
kishimoto81, tomiyama92, zel68}.

\begin{thm}\label{tomiyamatheorem}
If $s$ is a homemorphism of a compact and Hausdorff
topological space $X$, then the
following assertions are equivalent:
\begin{enumerate}[{\rm (i)}]
\item $(X,s)$ is topologically free;

\item if $I$ is a nonzero closed ideal of $C^*(X,s)$,
then $I \cap C(X) \neq \{ 0 \}$;

\item $C(X)$ is a maximal abelian $C^*$-subalgebra
of $C^*(X,s)$.
\end{enumerate}
\end{thm}

Inspired by Theorem \ref{tomiyamatheorem},
Svensson, Silvestrov and de Jeu \cite{svesildejeu09}
have shown an analogous result (see Theorem \ref{maintheorem})
relating properties of an arbitrary topological dynamical system $(X,s)$,
where $s$ is a homeomorphism,
to ideal properties of the skew group algebra
$C(X) \rtimes^{\sigma} {\Bbb Z}$ (note that in the literature
this algebra is often referred to as a {\it crossed product}).
Recall that $C(X) \rtimes^{\sigma} {\Bbb Z}$ is
defined as the collection
of formal sums $\sum_{n \in {\Bbb Z}} f_n u_n$, for $f_n \in C(X)$,
where $f_n = 0$ for all but finitely many $n \in {\Bbb Z}$.
Define addition and multiplication on $C(X) \rtimes^{\sigma} {\Bbb Z}$ by
\begin{equation}\label{zaddition}
\left( \sum_{n \in {\Bbb Z}} f_n u_n \right) +
\left( \sum_{n \in {\Bbb Z}} g_n u_n \right)=
\sum_{n \in {\Bbb Z}} (f_n+g_n) u_n
\end{equation}
respectively
\begin{equation}\label{zproduct}
\left( \sum_{n \in {\Bbb Z}} f_n u_n \right)
\left( \sum_{n \in {\Bbb Z}} g_n u_n \right)=
\sum_{n \in {\Bbb Z}} \left( \sum_{m \in {\Bbb Z}} f_m \sigma(m) (g_{n-m}) \right) u_n
\end{equation}
for $\sum_{n \in {\Bbb Z}} f_n u_n$, $\sum_{n \in {\Bbb Z}} g_n u_n \in C(X) \rtimes^{\sigma} {\Bbb Z}$.
Note that if $X$ is compact and Hausdorff, so that
$C^*(X,s)$ is defined, then
the skew group algebra $C(X) \rtimes^{\sigma} {\Bbb Z}$ is
a norm dense complex subalgebra of the $C^*$-algebra $C^*(X,s)$.

\begin{thm}[Svensson, Silvestrov and de Jeu \cite{svesildejeu09}]\label{maintheorem}
Suppose that $s$ is a homeomorphism of a Tychonoff and
Baire topological space $X$. The following assertions are equivalent:
\begin{enumerate}[{\rm (i)}]
\item $(X,s)$ is topologically free;

\item if $I$ is a nonzero ideal of $C(X) \rtimes^{\sigma} {\Bbb Z}$,
then $I \cap C(X) \neq \{ 0 \}$;

\item $C(X)$ is a maximal abelian complex subalgebra of $C(X) \rtimes^{\sigma} {\Bbb Z}$.
\end{enumerate}
\end{thm}

In fact Svensson, Silvestrov and de Jeu loc. cit. prove a more general statement
for skew group algebras $A \rtimes^{\sigma} {\Bbb Z}$
defined by complex subalgebras $A$ of $C(X)$ satisfying certain
extra conditions, see Theorem 4.5 in \cite{svesildejeu09} for more details.
For related results by the same authors,
see \cite{svesildejeu07} and \cite{svesildejeu08}.

In the above discussion the dynamics is
generated by a single continuous function.
Therefore, all powers of this function
commute which, in particular, means that
we have an abelian action on the space.
This raises the question if there is
a nonabelian version of Theorem 2.
Another natural question is if there
is a version of the same theorem that holds
for dynamics defined by families of
partial functions on a space, that is
functions that do not necessarily
have the same domain or codomain.
In this article we simultaneously
address both of these questions by
introducing category dynamical systems.
These are defined by families, stable under composition,
of continuous maps between potentially
different topological spaces.
We show generalizations of Theorem \ref{maintheorem}
to the case of skew category algebras defined by these maps
and spaces (see Theorem \ref{genmaintheorem} below and
Theorem \ref{manyimplications}, Theorem \ref{equivalent} and
Theorem \ref{discrete} in Section \ref{catdynsyst}).
To be more precise, suppose that $G$ is a category.
The family of objects of $G$ is denoted by $\ob(G)$;
we will often identify an object in $G$ with
its associated identity morphism.
The family of morphisms in $G$ is denoted by $\mor(G)$;
by abuse of notation, we will often write $n \in G$
when we mean $n \in \mor(G)$.
Throughout the article $G$ is assumed to be small, that is
with the property that $\mor(G)$ is a set.
The domain and codomain of a morphism $n$ in $G$ is denoted by
$d(n)$ and $c(n)$ respectively.
We let $G^{(2)}$ denote the collection of composable
pairs of morphisms in $G$, that is all $(m,n)$ in
$\mor(G) \times \mor(G)$ satisfying $d(m)=c(n)$.
For each $e \in \ob(G)$, let $G_e$ denote the
collection of $n \in \mor(G)$ with $d(n)=c(n)=e$.
Let $G^{\op}$ denote the opposite category of $G$.
We let $\Top$ denote the category
having topological spaces as objects
and continuous functions as morphisms.
Suppose that $s : G \rightarrow \Top^{\op}$ is a (covariant) functor;
in that case we say that $s$ is a {\it category dynamical system}.
If $G$ is a groupoid, that is a category
where all morphisms are isomorphisms, then
we say that $s$ is a {\it groupoid dynamical system}.
If $e \in \ob(G)$, then
we say that an element $x \in s(e)$ is \emph{periodic}
if there is a nonidentity $n : e \rightarrow e$ in $G$
such that $s(n)(x)=x$;
an element of $s(e)$ which is not periodic
is called \emph{aperiodic}.
We say that $s$ is {\it topologically free} if for each $e \in \ob(G)$,
the set of aperiodic elements of $s(e)$ is dense in $s(e)$.
For each $e \in \ob(G)$, we
let $C(e)$ denote the set of continuous complex valued
functions on $s(e)$.
For each $n \in G$ the functor $s$ induces a map
$\sigma(n) : C(d(n)) \rightarrow C(c(n))$
by the relation $\sigma(n)(f) = f \circ s(n)$,
for $f \in C(d(n))$.
If we use the terminology introduced in Section \ref{skewcategoryalgebras},
then the map $\sigma$ defines a {\it skew category system}
(see Definition \ref{defskewcategorysystem}).
In this section,
we show how one to
each skew category system may associate a so called
{\it skew category algebra} $A \rtimes^{\sigma} G$ (see also \cite{oinlun08}
for a more general construction)
where $A = \oplus_{e \in \ob(G)} C(e)$.
In Section \ref{skewcategoryalgebras},
we also investigate the connection between
maximal commutativity of $A$ in $A \rtimes^{\sigma} G$
and the way in which ideals intersect $A$ (see Proposition \ref{eqconditions}).
In Section \ref{catdynsyst}, we apply these results
to obtain connections between topological freeness of the category dynamical system $s$,
ideal properties of $A \rtimes^{\sigma} G$
and maximal commutativity of $A$
in $A \rtimes^{\sigma} G$ (see Theorems \ref{manyimplications}, \ref{equivalent} and \ref{discrete}).
In particular, in the end of the article,
we show the following generalization
of Theorem \ref{maintheorem} from the additive group of
integers to a large class of groupoids.

\begin{thm}\label{genmaintheorem}
Suppose that $s : G \rightarrow \Top^{\op}$ is a groupoid
dynamical system with the property
that for each $e \in \ob(G)$, the group of all
morphisms $e \rightarrow e$
is countable and $s(e)$ is Tychonoff and Baire.
If we put $A = \oplus_{e \in \ob(G)} C(e)$,
then the following assertions are equivalent:
\begin{enumerate}[{\rm (i)}]
\item $s$ is topologically free;

\item if $I$ is a nonzero ideal of $A \rtimes^{\sigma} G$, then $I \cap A \neq \{ 0 \}$;

\item the ring $A$ is a maximal abelian complex subalgebra of $A \rtimes^{\sigma} G$.
\end{enumerate}
\end{thm}

\section{Skew Category Algebras}\label{skewcategoryalgebras}

In this section, we define skew category systems
$(A,G,\sigma)$ (see Definition \ref{defskewcategorysystem})
and skew category algebras $A \rtimes^{\sigma} G$
(see Definition \ref{defskewcategoryalgebra}).
Thereafter, we investigate the existence of
identity elements in skew category algebras
(see Proposition \ref{unit}).
In the end of this section, we investigate
the relation between maximal commutativity
of $A$ and the ideal
intersection property
in skew category algebras $A \rtimes^{\sigma} G$ (see Proposition \ref{commutativeA}
and Proposition \ref{eqconditions}).
These results will be applied to
category dynamical systems in Section \ref{catdynsyst}.

\begin{conventions}
Let $R$ be an associative ring.
The identity map $R \rightarrow R$ is denoted by $\id_R$.
If $R$ is unital then the identity element of $R$
is nonzero and is denoted by $1_R$.
The category of unital rings is denoted by $\Ring$.
We say that a subset $R'$ of $R$ is a subring of $R$
if it is itself a ring under the binary operations of $R$.
We always assume that ring homomorphisms between
unital rings respect the identity elements.
If $A$ is a subset of $R$, then the {\it commutant} of $A$
in $R$ is the set of elements of $R$ that commute
with every element of $A$.
If $A$ is a commutative subring of $R$, then
$A$ is called {\it maximal commutative} in $R$
if the commutant of $A$ in $R$ equals $A$.
All ideals of rings are supposed to be two-sided.
If $R$ is commutative and $x \in R$, then $\Ann(x)$
denotes the ideal of $R$ consisting of all $y \in R$
satisfying $xy = 0$.
\end{conventions}

\begin{defn}\label{defskewcategorysystem}
By a \emph{skew category system} we mean a
triple $(A,G,\sigma)$ where $G$ is a (small) category,
$A$ is the direct sum of unital rings $A_e$, for $e \in \ob(G)$,
and $\sigma$ is a functor $G \rightarrow \Ring$ satisfying
$\sigma(n) : A_{d(n)} \rightarrow A_{c(n)}$, for $n \in G$.
\end{defn}

\begin{rem}\label{intermsofmaps}
Suppose that $(A,G,\sigma)$ is a skew category system.
The fact that $\sigma$ is a functor $G \rightarrow \Ring$
can be formulated in terms of maps by saying that
\begin{equation}\label{idd}
\sigma(e) = {\rm id}_{A_e},
\end{equation}
for all $e \in \ob(G)$, and
\begin{equation}\label{algebraa}
\sigma(m) \sigma(n) = \sigma(mn),
\end{equation}
for all $(m,n) \in G^{(2)}$.
\end{rem}

\begin{defn}\label{defskewcategoryalgebra}
If $(A,G,\sigma)$ is a skew category system, then
we let $A \rtimes^{\sigma} G$ denote the collection of formal sums
$\sum_{n \in G} a_n u_n$, where $a_n \in A_{c(n)}$, $n \in G$,
are chosen so that all but finitely many of them are nonzero.
Define addition and multiplication on $A \rtimes^{\sigma} G$ by
\begin{equation}\label{addition}
\left( \sum_{n \in G} a_n u_n \right) +
\left( \sum_{n \in G} b_n u_n \right) =
\sum_{n \in G} \left( a_n + b_n \right) u_n
\end{equation}
respectively
\begin{equation}\label{multiplication}
\left( \sum_{n \in G} a_n u_n \right)
\left( \sum_{n \in G} b_n u_n \right) =
\sum_{n \in G} \left( \sum_{\stackrel{(m,m') \in G^{(2)};}{mm' = n}}
a_m \sigma(m)(b_{m'}) \right) u_n
\end{equation}
for $\sum_{n \in G} a_n u_n$, $\sum_{n \in G} b_n u_n \in A \rtimes^{\sigma} G$.
It is clear that these operations define
a (sometimes nonunital, see Proposition \ref{unit})
ring structure on $A \rtimes^{\sigma} G$.
We call $A \rtimes^{\sigma} G$ the \emph{skew category algebra}
defined by $(A,G,\sigma)$.
\end{defn}

\begin{rem}
If $G$ is a groupoid, then (\ref{multiplication})
can be rewritten in the following slightly simpler form
\begin{equation}\label{groupoidproduct}
\left( \sum_{n \in G} a_n u_n \right)
\left( \sum_{n \in G} b_n u_n \right) =
\sum_{n \in G} \left( \sum_{\stackrel{m \in G;}{c(m)=c(n)}}
a_m \sigma(m)(b_{m^{-1}n}) \right) u_n.
\end{equation}
If $G$ equals the additive group ${\Bbb Z}$,
then (\ref{groupoidproduct}) coincides with (\ref{zproduct}).
\end{rem}

\begin{rem}
Suppose that $T := A \rtimes^{\sigma} G$ is a skew category algebra.
If we for each $n \in G$, put $T_n = A_{c(n)} u_n$,
then $T = \oplus_{n \in G} T_n$,
$T_m T_n = T_{mn}$, for $(m,n) \in G^{(2)}$,
and $T_m T_n = \{ 0 \}$, otherwise.
In the language of \cite{liu06}, \cite{lu06} and \cite{oinlun10}
this means that a skew category algebra is a strongly category
graded ring.
\end{rem}

\begin{defn}
Given a nonzero element $x = \sum_{n \in G} a_n u_n$
in the skew category algebra $A \rtimes^{\sigma} G$,
we put $$\dom(x) = \{ d(n) \mid n \in G \ {\rm and} \ a_n \neq 0 \}$$ and
$$\cod(x) = \{ c(n) \mid n \in G \ {\rm and} \ a_n \neq 0 \}.$$
\end{defn}

\begin{prop}\label{unit}
Given a nonzero element $x$ in the skew category algebra $A \rtimes^{\sigma} G$,
there are unique elements $l(x)$ and $r(x)$ in $A \rtimes^{\sigma} G$
subject to the condition that $l(x)y = y$ and $z r(x) = z$ for all
elements $y$ and $z$ in $A \rtimes^{\sigma} G$ satisfying $\cod(x)=\cod(y)$
and $\dom(x)=\dom(z)$, namely $l(x) = \sum_{e \in \cod(x)} 1_{A_e} u_e$
and $r(x) = \sum_{e \in \dom(x)} 1_{A_e} u_e$.
In particular, this implies that $A \rtimes^{\sigma} G$ is unital if and
only if $\ob(G)$ is finite;
in that case, the multiplicative identity of $A \rtimes^{\sigma} G$
is $\sum_{e \in \ob(G)} 1_{A_e} u_e$.
\end{prop}

\begin{proof}
Suppose that $x = \sum_{n \in G} a_n u_n$
is an element of $A \rtimes^{\sigma} G$
satisfying $l(x)y = y$ and $z r(x) = z$ for all
elements $y$ and $z$ in $A \rtimes^{\sigma} G$
with $\cod(x)=\cod(y)$ and $\dom(x)=\dom(z)$.
Suppose that $l(x) = \sum_{n \in G} b_n u_n$
for some $b_n \in A_{c(n)}$ satisfying $b_n = 0$
for all but finitely many $n \in G$.
If we put $y = \sum_{e \in \cod(x)} 1_{A_e} u_e$,
then $\cod(x)=\cod(y)$. Hence, by the assumptions,
we get that
$\sum_{e \in \cod(x)} 1_{A_e} u_e = y = l(x)y =
\left( \sum_{n \in G} b_n u_n \right)
\left( \sum_{e \in \cod(x)} 1_{A_e} u_e \right) =
\sum_{e \in \cod(x)} \sum_{n \in G, d(n)=e} b_n u_n$.
This implies that $b_n = 1_{A_{c(n)}}$, for $n \in \cod(x)$,
and $b_n = 0$, otherwise.
This means that $l(x) = \sum_{e \in \cod(x)} 1_{A_e} u_e$.
The statement concerning $r(x)$ is treated similarly.
On the other hand, it is clear that $l(x)$ and $r(x)$
as defined above satisfy the required property.
The last part follows immediately from the above.
\end{proof}


\begin{prop}\label{commutativeA}
If $A \rtimes^{\sigma} G$ is a skew
category algebra with $A$ commutative, then
\begin{enumerate}[{\rm (a)}]
\item the commutant of $A$ in $A \rtimes^{\sigma} G$
is the collection of elements of $A \rtimes^{\sigma} G$
of the form
$\sum_{e \in \ob(G)} \sum_{n \in G_e} a_n u_n$
satisfying $\sigma(n)(a)-a \in \Ann(a_n)$, for all $e \in \ob(G)$,
all $n \in G_e$ and all $a \in A_e$;

\item the subring $A$ is maximal commutative in
$A \rtimes^{\sigma} G$ if and only if
there for all $e \in \ob(G)$,
all nonidentity $n \in G_e$ and all nonzero $a_n \in A_e$,
is a nonzero $a \in A_e$ with the
property that $\sigma(n)(a)-a \notin \Ann(a_n)$;

\item if every $A_e$, for $e \in \ob(G)$, is an integral domain,
then $A$ is maximal commutative in $A \rtimes^{\sigma} G$ if and only if
there for every $e \in \ob(G)$ and every nonidentity $n \in G_e$
is $a \in A_e$ with the property that $\sigma(n)(a) \neq a$.
\end{enumerate}
\end{prop}

\begin{proof}
Let the commutant of $A$ in $A \rtimes^{\sigma} G$
be denoted by $A'$.

(a) Take $x = \sum_{n \in G} a_n u_n$ in $A'$.
From the fact that the equality $1_{A_e}u_e x = x 1_{A_e}u_e$
holds for all $e \in \ob(G)$, it follows that
$a_n = 0$ whenever $d(n) \neq c(n)$.
Take $e \in \ob(G)$ and $a \in A_e$.
From the fact that the equality $au_e x = x au_e$ holds
it follows that if $n \in G_e$, then
$a a_n = \sigma(n)(a) a_n$, that is
that $a - \sigma(n)(a) \in \Ann(a_n)$.
On the other hand,
it is clear that elements of $A \rtimes_{\alpha}^{\sigma} G$
of the form $\sum_{e \in \ob(G)} \sum_{n \in G_e} a_n u_n$
satisfying $\sigma(n)(a)-a \in \Ann(a_n)$, for $e \in \ob(G)$,
$n \in G_e$ and $a \in A_e$, belongs to $A'$.

(b) follows from (a) and (c) follows from (b).
\end{proof}

\begin{prop}\label{eqconditions}
Suppose that $A \rtimes^{\sigma} G$ is a skew
category algebra with $A$ commutative. Consider the following assertions:
\begin{enumerate}[{\rm (i)}]
\item if $I$ is a nonzero ideal of $A \rtimes^{\sigma} G$,
then $I \cap A \neq \{ 0 \}$;

\item the subring $A$ is maximal commutative in $A \rtimes^{\sigma} G$.
\end{enumerate}
Then:
\begin{enumerate}[{\rm (a)}]
\item {\rm (i)} implies {\rm (ii)};

\item {\rm (ii)} does not imply {\rm (i)} for all categories $G$;

\item if $G$ is a groupoid, then {\rm(i)} holds if and only if {\rm (ii)} holds.
\end{enumerate}
\end{prop}

\begin{proof}
(a) We show the contrapositive statement. Suppose that $A$ is not
maximal commutative in $A \rtimes^{\sigma} G$.
Then, by Proposition \ref{commutativeA}(b),
there exists some $e \in \ob(G)$, some nonidentity
$n \in G_e$ and some nonzero $a_n \in A_e$
with the property that $\sigma(n)(a)-a \in \Ann(a_n)$
for all $a \in A_e$, that is such that
$a_n \sigma(n)(a) = a_n a$ for all $a \in A_e$.
Let $I$ be the nonzero ideal of $A \rtimes^{\sigma} G$ generated
by the element $a_n u_e - a_n u_n$
and define the homomorphism of abelian groups
$\varphi : A \rtimes^{\sigma} G \rightarrow A$
by the additive extension of the relation
$\varphi(x u_t) = x$, for $t \in G$ and $x \in A_{c(t)}$.
We claim that $I \subseteq \ker(\varphi)$.
If we assume that the claim holds, then,
since $\varphi |_A = {\rm id}_A$,
it follows that
$A \cap I = \varphi |_A (A \cap I)
\subseteq \varphi(I) = \{ 0 \}$.
Now we show the claim.
By the definition of $I$ it follows that
it is enough to show that $\varphi$
maps elements of the form
$x u_r (a_n u_e - a_n u_n) y u_t$ to zero,
where $x \in A_{c(r)}$, $y \in A_{c(t)}$
and $r,t \in G$ satisfy $d(r)=e=c(t)$.
However, since $a_n \sigma(n)(a) = a_n a$ for all $a \in A_e$,
we get that
$x u_r (a_n u_e - a_n u_n) y u_t =
x u_r (a_n y u_t - a_n y u_n u_t)=
x u_r (a_n y u_t - a_n y u_{nt})=
x \sigma(r)(a_n y) u_{rt} - x \sigma(r)(a_n y) u_{rnt}$
which, obviously, is mapped to zero by $\varphi$.

(b) We will show that there even are monoids
for which (ii) does not imply (i).
Let ${\Bbb N}$ denote the nonnegative integers
equipped with addition as operation and zero as a neutral element.
Let $A$ denote the polynomial ring ${\Bbb C}[X]$
and fix a complex number $z$ which is not a root of
unity, that is such that $z^n \neq 1$ for all positive
integers $n$. For each $n \in {\Bbb N}$ define
$\sigma(n) : A \rightarrow A$ by
$\sigma(n)(p(X)) = p(z^n X)$, for $p(X) \in A$.
It is clear that each $\sigma(n)$, for $n \in {\Bbb N}$,
is a ring endomorphism of $A$.
It is easy to check that if $p(X)$ is any nonconstant
polynomial and $n$ is a positive integer,
then $\sigma(n)(p(X)) \neq p(X)$.
Since $A$ is an integral domain, this implies,
by Proposition \ref{commutativeA}(c), that $A$
is maximal commutative in $A \rtimes^{\sigma} G$.
However, if we let $I$ denote the ideal generated
by $u_1$, then it is clear that $A \cap I = \{ 0 \}$.

(c) By (a) it follows that (i) implies (ii).
Now we show the contrapositive statement of (ii) implies (i).
Let $C$ denote the commutant
of $A$ in $A \rtimes^{\sigma} G$ and suppose that $I$ is a
twosided ideal of $A \rtimes^{\sigma} G$ with the property that $I
\cap C = \{ 0 \}$. We wish to show that $I = \{ 0 \}$. Take $x \in
I$. If $x \in C$, then by the assumption $x = 0$. Therefore we now
assume that $x  = \sum_{s \in G} a_s u_s  \in I$, $a_s \in
A_{c(s)}$, $s \in G$, and that $x$ is chosen so that $x \notin C$
with the set $S := \{ s \in G \mid a_s \neq 0 \}$ of least possible
cardinality $N$. Seeking a contradiction, suppose that $N$ is
positive. First note that there is $e \in c(x)$ with $u_e x \in I
\setminus C$. In fact, if $u_e x \in C$ for all $e \in c(x)$,
then $x = l(x)x = \sum_{e \in c(x)} u_e x \in C$ which is a
contradiction. By minimality of $N$ we can assume that $c(s)=e$, $s
\in S$, for some fixed $e \in \ob(G)$. Take $t \in S$ and consider
the nonzero element $x' := x u_{t^{-1}} \in I$. Since
$I \cap C = \{ 0 \}$, we get that $x' \in I \setminus C$. Take $a =
\sum_{f \in \ob(G)} b_f u_f \in A$. Then $I \ni x'' := ax' - x'a =
\sum_{s \in S} a_s(b_{d(s)} - \sigma_s(b_e)) u_s $. Since the
summand for $s = e$ vanishes, we get, by the assumption on $N$, that
$x'' = 0$. Since $a \in A$ was arbitrarily chosen, we get that $x'
\in C$ which is a contradiction. Therefore $N = 0$ and hence $S =
\emptyset$ which in turn implies that $x=0$. Since $x \in I$ was
arbitrarily chosen, we finally get that $I = \{ 0 \}$.
\end{proof}

\begin{rem}
For other results related to the implication
(i) implies (ii) in Proposition \ref{eqconditions},
see \cite{oinlun08}.
The implication (ii) implies (i) in Proposition \ref{eqconditions}
actually holds for all nondegenerate groupoid graded rings,
see \cite{oinlun10}.
\end{rem}

\section{Category Dynamical Systems}\label{catdynsyst}

In this section, we show generalizations of
Theorem \ref{maintheorem} to the case of
skew category algebras defined by \emph{category
topological dynamical systems}
(see Theorem \ref{manyimplications}, Theorem \ref{equivalent}
and Theorem \ref{discrete}).
We use these results to prove Theorem \ref{genmaintheorem}.
To this end, we apply results from the
previous section to show results about
the commutant of $A$ in the
corresponding skew category algebra $A \rtimes^{\sigma} G$
(see Proposition \ref{maxcomcondition}).
In the end of this section, we discuss the implications
of these results for the connection between
dynamical systems defined by partially defined functions
(see Definition \ref{partialdefn}) and properties
of the corresponding skew category algebras
(see Example \ref{partialexmp}).

\begin{conventions}
Let $s : G \rightarrow \Top^{\op}$ be a category dynamical system
and suppose that $R$ denotes a commutative unital topological ring,
that is a commutative unital ring
which is also a topological space such that both
the addition and the multiplication are continuous as maps
$R \times R \rightarrow R$, where
$R \times R$ carries the product topology.
Take $e \in \ob(G)$. We let
$C(s(e),R)$ denote the $R$-algebra of continuous
maps $s(e) \rightarrow R$;
we let $1_e$ denote the
map $s(e) \rightarrow R$ that sends each $x \in s(e)$ to $1_R$;
if $R$ equals the complex numbers,
then we let $C(s(e),R)$ be denoted by $C(e)$.
For each $n \in G$, we let
$\sigma(n)$ denote the function
$C(s(d(n)),R) \rightarrow C(s(c(n)),R)$
defined by $\sigma(n)(f) = f \circ s(n)$,
for $f \in C(s(d(n)),R)$.
For the rest of the article, we suppose that
\begin{enumerate}[(i)]
\item $A_e$ denotes a subring of $C(s(e),R)$;

\item if $n \in G$ and $f \in A_{d(n)}$, then
$\sigma(n)(f)$ belongs to $A_{c(n)}$;

\item $A$ denotes the direct sum $\oplus_{e \in \ob(G)}A_e$.
\end{enumerate}
The \emph{support} of $f\in C(s(e),R)$ is denoted by $\Supp(f)$
and is the set of $x \in s(e)$ such that $f(x) \neq 0$.
Furthermore, we say that
a nonempty subset of $s(e)$ is a {\it domain
of uniqueness} for $A_e$ if every function of $A_e$
that vanishes on it, vanishes on the whole of $s(e)$.
We also say that the ring $A_e$ {\it separates elements} of $s(e)$
if for each pair of distinct elements $x$ and $y$ in $s(e)$
there exists a function $f$ in $A_e$ such that $f(x) \neq f(y)$.
\end{conventions}

\begin{prop}
The triple $(A , G , \sigma)$
is a skew category system.
\end{prop}

\begin{proof}
We need to check conditions (\ref{idd})
and (\ref{algebraa}) from Remark \ref{intermsofmaps}.
Take $e \in \ob(G)$ and $f \in A_e$.
Then $\sigma(e)(f) = f \circ s(e) = f$.
Therefore $\sigma(e) = \id_{A_e}$.
Take $(m,n) \in G^{(2)}$ and $f \in A_{d(n)}$.
Then $\sigma(m)\sigma(n)(f) = \sigma(m)( f \circ s(n) )=
f \circ s(n) \circ s(m) = f \circ (s(m) \circ_{\op} s(n))=
f \circ s(mn) = \sigma(mn)(f)$.
Therefore $\sigma(m) \sigma(n) = \sigma(mn)$.
\end{proof}

\begin{defn}
If $e \in \ob(G)$, $x \in s(e)$ and $n \in G_e$, then we put
\begin{enumerate}[(i)]
\item $\Sep_{A_e}^n = \{ x \in s(e) \mid$ there is $f \in A_e$
with $f(x) \neq \sigma(n)(f)(x) \}$;

\item $\Per_{A_e}^n = \{ x \in s(e) \mid$ for all $f \in A_e$,
$f(x) = \sigma(n)(f)(x)$ holds$\}$;

\item $\Sep_e^n = \{ x \in s(e) \mid x \neq s(n)(x) \}$;

\item $\Per_e^n = \{ x \in s(e) \mid x = s(n)(x) \}$;

\item $\Per_{A_e}^{\infty} = \bigcap_{n \in G_e \setminus \{ e \} } \Sep_{A_e}^n$;

\item $\Per_e^{\infty} = \bigcap_{ n \in G_e \setminus \{ e \} } \Sep_e^n$;

\item $\orb_e(x) = \{ s(n)(x) \mid n \in G_e \}$.
\end{enumerate}
\end{defn}

\begin{prop}\label{inclusions}
If $e \in \ob(G)$ and $n \in G_e$, then
\begin{enumerate}[{\rm (a)}]
\item $s(e) \setminus \Sep_{A_e}^n = \Per_{A_e}^n$;

\item $s(e) \setminus \Sep_e^n = \Per_e^n$;

\item $s(e) \setminus \Per_{A_e}^n = \Sep_{A_e}^n$;

\item $s(e) \setminus \Per_e^n = \Sep_e^n$;

\item $\Per_e^n \subseteq \Per_{A_e}^n$;

\item $\Sep_{A_e}^n \subseteq \Sep_e^n$;

\item $\Per_e^{\infty} \subseteq \Sep_e^n$.
\end{enumerate}
\end{prop}

\begin{proof}
(a), (b) and (e) are trivial. (c), (d) and (f) follow from (a), (b) and (e)
by taking complements. (g) follows from the definition of $\Per_e^{\infty}$.
\end{proof}

\begin{prop}\label{separates}
If $e \in \ob(G)$ and $n \in G_e$, then
\begin{enumerate}[{\rm (a)}]
\item if $A_e$ separates the elements of $s(e)$,
then $\Per_e^n = \Per_{A_e}^n$ and
$\Sep_e^n = \Sep_{A_e}^n$;

\item if $s(e)$ is a Hausdorff topological space
and $s(n)$ is continuous,
then $\Per_e^n$ is closed and $\Sep_e^n$ is open;

\item if $s(e)$ is a topological space
and $A_e$ is an $R$-subalgebra of $C(s(e),R)$ with the
property that each set $f^{-1}( \{ 0 \} )$, for
$f \in A_e$, is closed, then
$\Per_{A_e}^n$ is closed and $\Sep_{A_e}^n$ is open.
\end{enumerate}
\end{prop}

\begin{proof}
(a) The inclusion $\Per_e^n \subseteq \Per_{A_e}^n$ follows
from Proposition (\ref{inclusions})(e). The inclusion
$\Per_e^n \supseteq \Per_{A_e}^n$ follows from the
fact that $A_e$ separates the elements of $s(e)$.
The equality $\Sep_e^n = \Sep_{A_e}^n$ follows from the
equality $\Per_e^n = \Per_{A_e}^n$
and Proposition \ref{inclusions}(c)(d).

(b) It is well known that the set of fixed points of
a continuous function on a Hausdorff topological space
is a closed set. Therefore, $\Per_e^n$ is closed.
By Proposition \ref{inclusions}(b) it follows that $\Sep_e^n$ is open.

(c) Take $f \in A_e$. By the assumptions it follows that the set
$$\{ x \in s(e) \mid f(x) = \sigma(n)(f)(x) \} = (f-\sigma(n)(f))^{-1}( \{ 0 \} )$$
is closed. Therefore the set
$$\Per_{A_e}^n = \bigcap_{f \in A_e} (f-\sigma(n)(f))^{-1}( \{ 0 \} )$$
is closed. The last part follows from Proposition \ref{inclusions}(c).
\end{proof}

\begin{prop}\label{maxcomcondition}
Suppose that $R$ is an integral domain.
\begin{enumerate}[{\rm (a)}]
\item The commutant of $A$ in $A \rtimes^{\sigma} G$
is the set of elements of $A \rtimes^{\sigma} G$
of the form $\sum_{e \in \ob(G)}\sum_{n \in G_e} f_n u_n$ satisfying
$\Supp(f_n) \subseteq \Per_{A_e}^n$, for $e \in \ob(G)$ and $n \in G_e$.
In particular, $A$ is maximal commutative in $A \rtimes^{\sigma} G$
if and only if for each $e \in \ob(G)$ and each
nonidentity $n \in G_e$, the set $\Sep_{A_e}^n$
is a domain of uniqueness for $A_e$.

\item If for each $e \in \ob(G)$, $A_e$ separates the elements
of $s(e)$, then the commutant of $A$ in $A \rtimes^{\sigma} G$
is the set of elements of $A \rtimes^{\sigma} G$
of the form $\sum_{e \in \ob(G)} \sum_{n \in G_e} f_n u_n$ satisfying
$\Supp(f_n) \subseteq \Per_e^n$, for $e \in \ob(G)$ and $n \in G_e$.
In particular, $A$ is maximal commutative in $A \rtimes^{\sigma} G$
if and only if for each $e \in \ob(G)$ and each
nonidentity $n \in G_e$, the set $\Sep_e^n$
is a domain of uniqueness for $A_e$.
\end{enumerate}
\end{prop}

\begin{proof}
(a) By Proposition \ref{commutativeA}(a) we get that
the commutant of $A$ in $A \rtimes^{\sigma} G$
is the set of elements $\sum_{e \in \ob(G)}\sum_{n \in G_e} f_n u_n$
in $A \rtimes^{\sigma} G$ satisfying
$\sigma(n)(g)-g \in \Ann(f_n)$ for all $g \in A_e$.
This means that $f_n(\sigma(n)(g)-g) = 0$
for all $g \in A_e$. From the fact that $R$
is an integral domain it follows that
if $x \in \Supp(f_n)$, then $\sigma(n)(g)(x)-g(x) = 0$,
for all $g \in A_e$, that is $x \in \Per_{A_e}^n$.
Therefore we get that
$\Supp(f_n) \subseteq \Per_{A_e}^n$, for $e \in \ob(G)$ and $n \in G_e$.
By Proposition \ref{commutativeA}(b) it follows that
$A$ is maximal commutative in $A \rtimes^{\sigma} G$
if and only if for each choice of $e \in \ob(G)$ and
nonidentity $n \in G_e$, we get that if $f_n \in A_e$ is nonzero
there exists $g \in A_e$
with $\sigma(n)(g)-g \notin \Ann(f_n)$, that is
such that $f_n (\sigma(n)(g)-g) \neq 0$.
This is in turn equivalent to the contrapositive statement that
to each choice of $e \in \ob(G)$ and
nonidentity $n \in G_e$, the equality $f_n|_{\Sep_{A_e}^n} = 0$
implies that $f_n = 0$, that is that $\Sep_{A_e}^n$
is a domain of uniqueness for $A_e$.

(b) This follows immediately from (a) and Proposition \ref{separates}(a).
\end{proof}

\begin{rem}
Proposition \ref{maxcomcondition} generalizes Theorem 3.3
and Theorem 3.5 in \cite{svesildejeu07}.
\end{rem}

\begin{thm}\label{manyimplications}
Consider the following three assertions:
\begin{enumerate}[{\rm (i)}]
\item $s$ is topologically free;

\item if $I$ is a nonzero ideal of $A \rtimes^{\sigma} G$,
then $I \cap A \neq \{ 0 \}$;

\item $A$ is maximal commutative in $A \rtimes^{\sigma} G$.
\end{enumerate}
Then:
\begin{enumerate}[{\rm (a)}]
\item {\rm (ii)} implies {\rm (iii)};

\item if $G$ is a groupoid, then {\rm (ii)} holds if and only if {\rm (iii)} holds;

\item if $R$ is an integral domain and for each $e \in \ob(G)$
the function space $A_e$ separates elements of $s(e)$ and for each $f \in A_e$,
the set $f^{-1}(\{ 0 \})$ is closed in $s(e)$, then {\rm (i)} implies {\rm (iii)};
this holds e.g. if $\{ 0 \}$ is closed in $R$, which, in turn, holds
e.g. when $R$ is Hausdorff;

\item if $R$ is an integral domain and for each $e \in \ob(G)$,
the space $s(e)$ is Hausdorff and Baire, the function
space $A_e$ separates elements of $s(e)$ and has the property that
for each nonempty open subset $U$ of $s(e)$, there is
a nonzero $f \in A_e$ that vanishes on $s(e) \setminus U$,
then {\rm (iii)} implies {\rm (i)}.
\end{enumerate}
\end{thm}

\begin{proof}
(a) and (b) follow immediately from
Proposition \ref{eqconditions}(a)(c).

(c) Take an $e \in \ob(G)$ and a nonidentity $n \in G_e$.
Suppose that the set $\Per_e^{\infty}$ is dense in $s(e)$.
By Proposition \ref{inclusions}(g), the inclusion
$\Per_e^{\infty} \subseteq \Sep_e^n$ holds.
By Proposition \ref{separates}(a) it follows that $\Sep_{A_e}^n$ is dense in $s(e)$.
If $f \in A_e$ satisfies $f(x) = 0$, for $x \in \Sep_{A_e}^n$,
then $f^{-1}(\{ 0 \})$ is a closed subset of $s(e)$ containing
the dense subset $\Sep_{A_e}^n$ of $s(e)$.
Therefore, $f^{-1}(\{ 0 \}) = s(e)$ and hence $\Sep_{A_e}^n$ is
a domain of uniqueness for $A_e$.
By Proposition \ref{maxcomcondition}(a) it follows
that $A$ is maximal commutative in $A \rtimes^{\sigma} G$.

(d) We show the contrapositive of (iii) implies (i).
Suppose that for some $e \in \ob(G)$, the
set $\Per_e^{\infty}$ is not dense in $s(e)$.
Since $s(e)$ is Hausdorff, it follows, by Proposition \ref{separates}(b),
that all the sets $\Sep_e^n$, for nonidentity $n \in G_e$, are open.
Therefore, since $s(e)$ is Baire,
there is some nonidentity $n \in G_e$
for which $\Sep_e^n$ is not dense, that is
such that $\Per_e^n$ has a nonempty interior $U$.
By the assumption on $A_e$ there is a nonzero
$f \in A_e$ that vanishes outside $U$.
Hence, by Proposition \ref{maxcomcondition}(b), it follows that
$f u_n$ belongs to the commutant of $A$ in $A \rtimes^{\sigma} G$.
But since obviously $f u_n$ does not belong to $A$,
the subring $A$ is not maximal commutative in $A \rtimes^{\sigma} G$.
\end{proof}

\begin{rem}
Theorem \ref{manyimplications} generalizes Theorem 3.7 in \cite{svesildejeu07}.
\end{rem}

\begin{rem}\label{topremark}
If $X$ is a Hausdorff and locally compact topological space,
then $X$ is Baire and all the complex algebras
\begin{itemize}
\item $C(X)$
\item $C_c(X) = \{ f \in C(X) \mid \Supp(f) \ {\rm compact} \}$
\item $C_b(X) = \{ f \in C(X) \mid f \ {\rm bounded} \}$
\item $C_0(X) = \{ f \in C(X) \mid {\rm for} \
{\rm every} \ \epsilon > 0, \ \{ x \mid |f(x)| \geq \epsilon \} \
{\rm is} \ {\rm compact} \}$
\end{itemize}
separate elements and have the property that to each nonempty open
subset $U$ of $X$, there is
a nonzero function $f$ in the algebra
that vanishes on $X \setminus U$ (see e.g. \cite{mun00}).
\end{rem}

\begin{thm}\label{equivalent}
Suppose that $s : G \rightarrow \Top^{\op}$
is a groupoid dynamical system.
If for each $e \in \ob(G)$ the topological space
$s(e)$ is Hausdorff and locally compact and
$A_e$ equals $C(e)$ (or $C_c(e)$; or $C_b(e)$;
or $C_0(e)$),
then the following three assertions are equivalent:
\begin{enumerate}[{\rm (i)}]
\item $s$ is topologically free;

\item if $I$ is a nonzero ideal of $A \rtimes^{\sigma} G$,
then $I \cap A \neq \{ 0 \};$

\item the ring $A$ is a maximal abelian complex subalgebra of
$A \rtimes^{\sigma} G$.
\end{enumerate}
\end{thm}

\begin{proof}
This follows from Theorem \ref{manyimplications}
and Remark \ref{topremark}.
\end{proof}

\begin{thm}\label{discrete}
Suppose that for each $e \in \ob(G)$
the sets $s(e)$ and $R$ carry the discrete
topologies and $A_e$ equals the set of
all functions $s(e) \rightarrow R$.
Consider the following three assertions:
\begin{enumerate}[{\rm (i)}]
\item $s$ is topologically free;

\item if $I$ is a nonzero ideal of $A \rtimes^{\sigma} G$,
then $I \cap A \neq \{ 0 \}$;

\item $A$ is maximal commutative in $A \rtimes^{\sigma} G$.
\end{enumerate}
Then:
\begin{enumerate}[{\rm (a)}]
\item {\rm (ii)} implies {\rm (iii)};

\item if $G$ is a groupoid, then {\rm (ii)} holds
if and only if {\rm (iii)} holds;

\item if $R$ is an integral domain, then
{\rm (i)} holds if and only if {\rm (iii)} holds;

\item if there is $e \in \ob(G)$
and $x \in s(e)$ such that $G_e$ has
cardinality greater than the cardinality of
$\orb_e(x)$ and $G_e$ is a divisible monoid, that is
such that for any distinct $m,n \in G_e$,
there is $p \in G_e$ with $m=np$
or $n = mp$, then {\rm (iii)} does not hold;

\item if there is $e \in \ob(G)$ and $x \in s(e)$ such that
$\orb_e(x)$ is finite and $G_e$ is an infinite and divisible monoid, e.g.
$({\Bbb Z},+)$ or $({\Bbb N},+)$, then
{\rm (iii)} does not hold.
\end{enumerate}
\end{thm}

\begin{proof}
(a), (b), and (c) follow immediately from Theorem \ref{manyimplications}.

(d) Suppose that $e \in \ob(G)$ and $x \in s(e)$
are chosen so that $G_e$ is divisible and has
cardinality greater than the cardinality of $\orb_e(x)$.
Then there exist distinct $m,n \in G_e$ with $s(m)(x)=s(n)(x)$.
Since $G_e$ is divisible, there is a nonidentity $p \in G_e$
with $m = np$ (or $n = mp$). This implies that
$s(np)(x)=s(n)(x)$ (or $s(mp)(x)=s(m)(x)$), that is
$s(p)s(n)(x) = s(n)(x)$ (or $s(p)s(m)(x) = s(m)(x)$).
Hence $s(n)(x) \in \Per_e^p$ (or $s(m)(x) \in \Per_e^p$)
which implies that $\Per_e^p$ is nonempty.
The claim now follows from Theorem \ref{manyimplications}(d).

(e) follows from (d).
\end{proof}

\begin{defn}\label{tychonoff}
Recall that a topological space $X$ is called {\it completely regular}
if given any open subset $U$ of $X$ and any $x \in U$,
there is $f \in C(X)$ with $f(x)=1$ and $f(y)=0$, for $y \in X \setminus U$.
Moreover, a topological space $X$ is called {\it Tychonoff}
if it is completely regular and Hausdorff.
\end{defn}

\noindent {\bf Proof of Theorem \ref{genmaintheorem}.}
This follows from Theorem \ref{equivalent}
and Definition \ref{tychonoff}. {\hfill $\square$}

\begin{defn}\label{partialdefn}
Suppose that $X$ is a topological space.
By a {\it partially defined dynamical system}
on $X$ we mean a collection $P$ of functions
such that
\begin{itemize}
\item if $f\in P$, then the domain $d(f)$ of $f$
and the codomain $c(f)$ of $f$ are subsets of $X$
and $f$ is continuous as a function
$d(f) \rightarrow c(f)$ where $d(f)$ and $c(f)$ are equipped
with the relative topologies induced by the topology on $X$;

\item if $f \in P$, then $\id_{d(f)} \in P$ and $\id_{c(f)} \in P$; 

\item if $f,g\in P$ are such that $d(f)=c(g)$, then $f \circ g\in P$.
\end{itemize}
We say that an element $x$ of $X$ is periodic with respect to $P$
if there is a nonidentity function $f$ in $P$
with $d(f)=c(f)$ and $f(x)=x$.
An element $x$ is aperiodic with respect to $P$
if it is not periodic.
We say that $P$ is topologically free if the set
of aperiodic elements of $X$ is dense in $X$.
By abuse of notation, we let $P$ denote the category
having the domains and codomains of functions in $P$
as objects and the functions of $P$ as morphisms.
We let the obvious functor $P \rightarrow \Top$
be denoted by $t_P$. Let $G_P$ denote the opposite
category of $P$ and let $s_P : G_P \rightarrow \Top^{\op}$
denote the opposite functor of $t_P$.
We will call $s_P$ the \emph{category dynamical system on $X$
defined by the partially defined dynamical system $P$}.
\end{defn}

\begin{prop}
If $X$ is a topological space and $P$ is a
partially defined dynamical system on $X$,
then $P$ is topologically free if and only if
$s_P$ is topologically free as a category
dynamical system.
\end{prop}

\begin{proof}
This follows immediately from the definition
of topological freeness of $P$ and $s_P$.
\end{proof}

To illustrate the above definitions and results,
we end the article with a simple example of
a partially defined dynamical system.

\begin{exmp}\label{partialexmp}
Suppose that we let $X$ denote the real numbers
equipped with its usual topology
and we let $Y$ denote the set of nonnegative real numbers equipped with
the relative topology induced by the topology on $X$.
Let $\opsqr : X \rightarrow Y$
and $\opsqrt : Y \rightarrow X$ denote the square function
and the square root function, respectively.
Furthermore, let $\opabs : X \rightarrow X$ denote the absolute value.
Let $P$ be the partially defined dynamical system
with $\ob(P) = \{ X,Y \}$ and $\mor(P) = \{ \id_X , \id_Y , \opsqr , \opsqrt , \opabs \}$.
Then we get the following table of partial composition for $P$
$$\begin{array}{c|ccccc}
\circ & \id_X & \id_Y & \opsqr & \opsqrt  & \opabs \\ \hline
\id_X & \id_X & *     & *   & \opsqrt  & \opabs \\
\id_Y & *     & \id_Y & \opsqr & *     & *   \\
\opsqr   & \opsqr   & *     & *   & \id_Y & \opsqr \\
\opsqrt  & *     & \opsqrt  & \opabs & *     & * \\
\opabs   & \opabs   & *     & *   & \opsqrt  & \opabs
\end{array}$$
Let $G = P^{\op}$ and put $A_X = C(X,{\Bbb R})$ and $A_Y = C(Y,{\Bbb R})$.
Take $f_X , f_X' , g_X , g_X' , h_X , h_X' \in A_X$ and
$f_Y , f_Y' , g_Y , g_Y' \in A_Y$. Then the product of
$$B_1 := f_X u_{\id_X} + g_X u_{\opabs} + h_X u _{\opsqr} + f_Y u_{\id_Y} + g_Y u_{\opsqrt}$$
and
$$B_2 := f_X' u_{\id_X} + g_X' u_{\opabs} + h_X' u _{\opsqr} + f_Y' u_{\id_Y} + g_Y' u_{\opsqrt}$$
in the skew category algebra $A \rtimes^{\sigma} G$ equals
$$
B_1 B_2 = f_X f_X' u_{\id_X} +
\left( f_X g_X' + g_X (f_X' \circ \opabs) + g_X (g_X' \circ \opabs) + h_X (g_Y' \circ \opsqr) \right) u_{\opabs} +$$
$$+ \left( f_X h_X' + h_X (f_Y' \circ \opsqr) + g_X (h_X' \circ \opabs) \right) u_{\opsqr} +
\left( f_Y f_Y' + g_Y (h_X' \circ \opsqrt) \right) u_{\id_Y} +$$
$$ + \left( f_Y g_Y' + g_Y (f_X' \circ \opsqrt) + g_Y (g_X' \circ \opsqrt) \right) u_{\opsqrt}.$$
Now we examine the properties (i), (ii) and (iii) in Theorem \ref{equivalent}.
It turns out that they are all false for our particular example.

Property (i) is false. In fact, the subset of periodic elements of $X$ with respect to $P$
is $Y$. Therefore, the set of aperiodic elements of $X$ is not dense in X.

Property (ii) is false. Indeed, the ideal $I$ of $A \rtimes^{\sigma} G$
generated by $u_{\opabs}$ equals
$$A_X u_{\opabs} + A_X u_{\opsqr} + A_Y u_{\opsqrt}$$
and hence has the property that $I \cap A = \{ 0 \}$.

Property (iii) is false. By Proposition \ref{maxcomcondition}(b) it follows
that the commutant of $A$ in $A \rtimes^{\sigma} G$ equals the set of elements
of $A \rtimes^{\sigma} G$ of the form
$$f_X u_{\id_X} + g_X u_{\opabs} + f_Y u_{\id_Y}$$
with $f_X , g_X \in A_X$ and $f_Y \in A_Y$
satisfying $\Supp(g_X) \subseteq Y$.
Therefore $A$ is not maximal commutative in $A \rtimes^{\sigma} G$.
\end{exmp}

\section*{Acknowledgements}
The second author was partially supported by The Swedish Research Council (postdoctoral fellowship no. 2010-918),
The Danish National Research Foundation (DNRF) through the Centre for Symmetry and Deformation,
The Crafoord Foundation,
The Royal Physiographic Society in Lund,
The Swedish Foundation for International Cooperation in Research and Higher Education (STINT) and The Swedish Royal Academy of Sciences.

\end{document}